\documentclass[11pt]{amsart}
\usepackage{amssymb,amscd,amsthm,verbatim}
\usepackage{amsbsy}
\usepackage{latexsym}
\usepackage[all,cmtip]{xy}
\usepackage[mathscr]{euscript}
\usepackage{enumitem}

\usepackage{etex}
\usepackage{graphicx}
\usepackage{pictex}
\usepackage{epstopdf}
\usepackage{color}

\usepackage{bbm}

\begin{document}

\theoremstyle{plain}
\newtheorem{thm}{Theorem}[]
\newtheorem{lem}{Lemma}[]
\newtheorem{prop}[thm]{Proposition}
\newtheorem*{cor}{Corollary}

\theoremstyle{definition}
\newtheorem{defn}{Definition}[]
\newtheorem{conj}{Conjecture}[]
\newtheorem{exmp}{Example}[]
\newtheorem*{qtn}{Question}
\newtheorem{cl}{Claim}

\theoremstyle{remark}
\newtheorem*{rem}{Remark}
\newtheorem*{pf}{Proof}
\newtheorem*{note}{Note}
\newtheorem{case}{Case}

\title[Cantor Spectrum]{Cantor Spectrum for CMV and Jacobi Matrices with Coefficients arising from Generalized Skew-Shifts}

\author[H.\ Jun]{Hyunkyu Jun}
\address{Department of Mathematics, Rice University, Houston, TX~77005, USA}
\email{hj12@rice.edu}

\begin{abstract}
We consider continuous cocycles arising from CMV and Jacobi matrices. Assuming the Verblunsky and Jacobi coefficients arise from generalized skew-shifts, we prove that uniform hyperbolicity of the associated cocycles is $C^0$-dense. This implies that the associated CMV and Jacobi matrices have Cantor spectrum for a generic continuous sampling map.
\end{abstract}

\maketitle

\section{Introduction}
Let $X$ be a compact metric space and let $T:X\to X$ be a strictly ergodic homeomorphism (i.e., $T$ is minimal and uniquely ergodic), which fibers over an almost periodic dynamical system (generalized skew-shifts). This means there exists an infinite compact abelian group $\mathbb G$ and an onto continuous map $f:X\to \mathbb G$ such that $T(f(x))=T(x)+g$ for some $g\in \mathbb G.$ We consider CMV matrices and Jacobi matrices whose Verblunsky coefficients and respectively, Jacobi coefficients are obtained by a continuous sampling map along an orbit of $T.$ Our interest is to investigate spectral properties.

By the nature of dynamically defined Verblunsky and Jacobi coefficients, our results rely on a connection between spectral properties and dynamics of linear cocycles, which is first established by Johnson \cite{MR818861}, often called Johnson's theorem. Roughly speaking, Johnson's theorem provides a connection between the spectrum of self-adjoint linear differential operators and uniform hyperbolicity of the associated linear cocycles referred to as ``an exponential dichotomy'' in Johnson \cite{MR818861}.

 Two similar results, which are directly connected to our work are Damanik \emph{et al} \cite{MR3543643} for CMV matrices and Marx  \cite{MR3291142} for Jacobi matrices. In \cite{MR3543643}, the authors show that the uniform spectrum of a CMV matrix consists of unimodular complex numbers whose associated cocycles are not uniformly hyperbolic. Likewise, in \cite{MR3291142}, the author proves that the uniform spectrum of a Jacobi matrix consists of energies whose associated cocycles are not uniformly hyperbolic.

In this paper, we consider the continuous cocycles arising from CMV and Jacobi matrices and show that uniform hyperbolicity is $C^0$-dense in both cases. Together with the results in \cite{MR3543643} and \cite{MR3291142}, this implies that 
the uniform spectrum of a CMV  or a Jacobi matrix is a Cantor set for a generic continuous sampling map.

Let us discuss a paper, which is intimately related to our work. In Avila \emph{et al} \cite{MR2477761}, the authors consider continuous $SL(2,\mathbb R)$-cocycles with the same base dynamics as in the present paper. The authors prove that if a cocycle is not uniformly hyperbolic and its homotopy class does not display a certain obstruction, it can be $C^0$-perturbed to become uniformly hyperbolic. Using this and ``a projection lemma'', which is also proved in \cite{MR2477761}, the authors show that uniform hyperbolicity is $C^0$-dense for the cocycles arising from Schr{\"o}dinger operators. In turn, the $C^0$-density implies Cantor spectrum for a generic continuous potential.

Our work fully utilizes their results on the general $SL(2,\mathbb R)$-cocycles. In addition, the proof of the Jacobi case is a direct application of the projection lemma. From our point of view, the applicability of the projection lemma is related to the solvability of a system of equations. We were unable to find a possible solvability for the case of CMV matrices. Thus, another constructive way of proof will be provided.

The spectral theory of Schr{\"o}dinger operators and Jacobi matrices with dynamically defined potentials and coefficients respectively, has been extensively studied for the past few decades in a variety of settings, e.g., random potentials, almost periodic potentials, subshift potentials, etc. (see \cite{MR3681983} and \cite{MR3719264} for surveys.) On the other hand, the case of CMV matrices is much less understood. Damanik and Lenz \cite{damanik2007uniform} consider ergodic families of Verblunsky coefficients generated by minimal aperiodic subshifts, thus a sampling map in \cite{damanik2007uniform} may be regarded as a simple function taking finitely many values.

Conspicuously absent while interesting was the case of almost periodic Verblunsky coefficients; see \cite{MR2105089}, pp. 706--707. Bourget \emph{et al} \cite{MR1962460} studied some almost periodic case but their model is modified so that it is distinguished from true CMV matrices.  Recently, Wang and Damanik \cite{MR3912798} consider quasiperiodic Verblunsky coefficients with analytic sampling maps and show Anderson localization in the regime of positive Lyapunov exponents. Also, in the forthcoming companion paper, the authors show that the spectrum is purely absolutely continuous at a small coupling. Our work considers generalized skew-shifts, which include the almost periodic case with continuous sampling maps.

\section{Statement of results}

Let $X$ be a compact metric space and let $T:X\to X$ be a homeomorphism. Given a continuous map $A:X\to SL(2, \mathbb R)$, a \emph{continuous cocycle} $(T,A):X\times \mathbb R^2\to X\times \mathbb R^2$ is defined as $(x,v)\to (T(x), A(x)v).$ For $n\in \mathbb Z$, $A^n$ is defined by $(T,A)^n=(T^n, A^n).$



 \begin{defn}

A continuous cocycle 
$$(T,A): X\times \mathbb R^2 \to X\times \mathbb R^2, (x,v)\to (T(x),A(x)v)$$
is \emph{uniformly hyperbolic} if there are $C>0$ and $\lambda<1$ and, for every $x\in X,$ there exist  one dimensional subspaces $E^s_x$ and $E^u_x$ of $\mathbb R^2$ such that 
\begin{enumerate}
\item $A(x)E^s_x=E^s_{T(x)}$ and $A(x)E^u_x=E^u_{T(x)}$
\item $||A^n(x)v^s||\leq C\lambda^n||v^s||$ and $||A^{-n}(x)v^u||\leq C\lambda^n||v^u||$
\end{enumerate}
for every $v^s\in E^s_x, v^u\in E^u_x,x\in X$ and $n\geq 1.$
\label{UH}
\end{defn}

Equivalently, it is well-known that
$(T,A)$ is uniformly hyperbolic if and only if there exist constants $c > 0$ and $\sigma > 1$ such that $||A^n(x)||>c\sigma^n$ for all $x\in X$ and $n\geq 1.$ (See \cite{MR3289050}, for example.)

In fact, $E^s_x$ and $E^u_x$ are unique if they exist and depend continuously on $x\in X.$ (compare \cite{MR3289050}.) Thus, we may choose a continuous map $u$ from $X$ to the unit circle in $\mathbb R^2$ such that $u(x)\in E^u_x.$

\subsection{CMV matrices}

Let $\mathbb D:=\{z\in \mathbb C: |z|<1\}.$ Let $\mu$ be a nontrivial probability measure on $\partial \mathbb D,$ i.e., it is not supported on a finite set. Then, we may define the $n$-th monic orthogonal polynomial $\Phi_n:=\Phi_n(z;d\mu)$ by $\Phi_n\perp z^l$ for $l=0,1,\cdots , n-1.$

Thus, we have $\langle \Phi_n, \Phi_m \rangle=0$ for all $m\neq n$ in $L^2(\partial \mathbb D, d\mu).$ Naturally, orthonormal polynomials $\phi_n$ are defined as $\phi_n(z)=\Phi(z)/||\Phi(z) ||.$

It is well known that the monic orthogonal polynomials are generated by the Szeg{\" o} recursion,
$$\Phi_{n+1}(z) = z\Phi _n(z) - \overline{\alpha_n}\Phi^*_n(z)$$
where $\{\alpha_0,\alpha_1,\alpha_2,\cdots \}\subset \mathbb D$ are suitably chosen parameters, called  Verblunsky coefficients. Conversely, given a sequence $\{\alpha_0,\alpha_1,\alpha_2,\cdots \}\subset \mathbb D$ we may define monic orthogonal polynomials with respect to a nontrivial probability measure on $\partial \mathbb D$ by the Szeg{\" o} recursion. In fact, Verblunsky's theroem says there is a one-to-one correspondence between nontrivial probability measures and sequences in $ \mathbb D$.

The standard CMV matrix associated to the measures $\mu$ is a matrix  representation discovered by Cantero \emph{et al} \cite{MR1955452} for multiplication by $z\in \partial \mathbb D$ in $L^2(\partial \mathbb D,d\mu)$. The matrix is given by: 

$$
\mathcal C=\begin{bmatrix}
    & \overline {\alpha_0} & \overline{\alpha_1}\rho_0 & \rho_1\rho_0  & & & &\\
   & \rho_0 & -\overline{\alpha_1}\alpha_0 & -\rho_1\alpha_0 & & & & \\
    &  & \overline{\alpha_2}\rho_1 &-\overline{\alpha_2}\alpha_1 & \overline{\alpha_3}\rho_2  &\rho_3\rho_2 & & &\\
        &  & \rho_2\rho_1 &-\rho_2\alpha_1 & -\overline{\alpha_3}\alpha_2  &-\rho_3\rho_2 & & &\\
            &  & &  & \overline{\alpha_4}\rho_3 & -\overline{\alpha_4}\alpha_3& \overline{\alpha_5}\rho_4 & \\
               &  & &  & \rho_4\rho_3 &- \rho_4\alpha_3& -\overline{\alpha_5}\rho_4 & \\
 &  & &  &  & \ddots&  \ddots& \ddots\\
\end{bmatrix}
$$
 where $\rho_n=(1-|\alpha_n|^2)^{1/2}.$ 
 
 The basis for the representation is obtained by orthonormalizing  
 $$\{1, z, z^{-1}, z^2, z^{-2},z^3, z^{-3},\cdots \}.$$ Note that the basis for the representation is not the orthonormal polynomials. The matrix representation based on the orthonormal polynomials is called GGT representation; see \cite{MR2105088}, Section 4.1.
 
 Let us briefly discuss why the CMV representation is a more suitable choice for spectral analysis. First of all, the set of orthonormal polynomials, $\{1,\phi_1,\phi_2,\cdots \},$ may not be a basis of $L^2(\partial \mathbb D, d\mu).$ Indeed, the orthornomal polynomials form a basis if and only if $\sum_{j=0}^\infty |\alpha_j|^2=\infty$ where $\alpha_j$'s are the corresponding Verblunsky coefficients (\cite{MR2105088}, Theorem 1.5.7). Even for the case when the orthonormal polynomials form a basis, a row of its GGT representation has infinitely many nonzero terms. The five diagonal form of a CMV matrix allows us to connect the solution $u$ of $\mathcal Cu =zu$, $z\in \partial \mathbb D,$ to $2\times 2$ matrices and this provides very useful tools for spectral analysis. On the other hand, we do not have this connection for GGT representation as its rows are not finite (see \cite{MR2105088}, Sections 4.1 and 4.2 for more discussion.)

Now, let us discuss CMV matrices over dynamical systems. Let $(X,\nu)$ be a probability measure space and let $T:X\to X$ be an invertible measure preserving transformation. Under this setting, we may consider dynamically defined Verblunsky coefficients with a measurable function $f:X\to \mathbb D$. That is, our coefficients $\{\alpha_n\}_{n\in \mathbb Z_+}$ are defined by $\alpha_n=f(T^n x)$ for some $x\in X.$ As $T$ is an invertible map, we may also consider a bi-infinite sequence, $\{\alpha_n\}_{n\in \mathbb Z}=\{f(T^n x) \}_{n\in \mathbb Z}.$ This leads to a bi-infinite CMV matrix, called an extended CMV matrix:

$$
\mathcal E_x=\begin{bmatrix}
    \ddots& \ddots & \ddots &   & \\
    & -\overline {\alpha_0}\alpha_{-1} & \overline{\alpha_1}\rho_0 & \rho_1\rho_0  & & & &\\
   & -\rho_0\alpha_{-1} & -\overline{\alpha_1}\alpha_0 & -\rho_1\alpha_0 & & & & \\
    &  & \overline{\alpha_2}\rho_1 &-\overline{\alpha_2}\alpha_1 & \overline{\alpha_3}\rho_2  &\rho_3\rho_2 & & &\\
        &  & \rho_2\rho_1 &-\rho_2\alpha_1 & -\overline{\alpha_3}\alpha_2  &-\rho_3\rho_2 & & &\\
            &  & &  & \overline{\alpha_4}\rho_3 & -\overline{\alpha_4}\alpha_3& \overline{\alpha_5}\rho_4 & \\
               &  & &  & \rho_4\rho_3 &- \rho_4\alpha_3& -\overline{\alpha_5}\rho_4 & \\
 &  & &  &  & \ddots&  \ddots& \ddots\\
\end{bmatrix}
$$

where, again, $\rho_n=(1-|\alpha_n|^2)^{1/2}$. 

Extended CMV matrices are useful tools to study spectral properties. Let us now assume that $T:X\to X$ is an ergodic invertible measure preserving transformation. With the associated extended CMV matrix, we have $\sigma(\mathcal E_x)=\sigma(\mathcal E_y)$ for $\nu$-almost every $x,y\in X.$ Moreover, the almost sure spectrum is purely essential. On the other hand, with the standard CMV matrix, what we obtain is that the essential spectrum coincides for $\nu$-almost every $x\in X$ and the discrete spectrum may depend on $x\in X.$ (\cite{MR2105089}, Theorem 10.16.1 and Theorem 10.16.2.) Moreover, we may draw more conclusions from Kotani theory with the extended CMV matrix. (\cite{MR2105089}, Theorem 10.11.1 -- 10.11.4.)

The spectrum of extended CMV matrices associated to the dynamically defined Verblunsky coefficients are closely related to the Szeg{\"o} cocycle defined as
$(T,\overline A_z(x))=(T,\overline A(f(x),z)), z\in \partial \mathbb D,$ where
$$\overline A(f(x),z):=\frac 1 {z^{1/2}\sqrt{1-|f(x)|^2}}\begin{bmatrix}
    z    & -\bar f(x)   \\
    -f(x)z     & 1
\end{bmatrix}.$$

$\overline A_z(x)$ is an element of ${SU}(1,1),$ which may not be in $SL(2,\mathbb R).$ However, there is a canonical conjugacy between $ {SU}(1,1)$ and $SL(2,\mathbb R)$, which we will explain later in more detail. 

Let $X$ be a compact metric space. If $T:X\to X$ is a minimal homeomorphism and $f\in C^0(X,\mathbb D)$, there is a uniform compact set $\Sigma\subset \partial \mathbb D$ with $\sigma(\mathcal E_x)=\Sigma$ for every $x\in X.$ Damanik \emph{et al} \cite{MR3543643} show that the uniform spectrum is given by $\Sigma=\partial \mathbb D\setminus U$, where
$$U=\{z\in \partial \mathbb D: (T, \overline A_z)\text{ is uniformly hyperbolic}\}.$$
Note that under the same hypothesis, the spectrum of standard CMV matrices may depend on $x\in X.$

Our strategy is to show that, given $\epsilon>0,$ if $(T, \overline A(f,z))$ is not uniformly hyperbolic, there exists $f'\in C^0(X,\mathbb D)$ such that $||\overline A(f,z)-\overline A(f',z)||_{C^0}<\epsilon$ and $(T, \overline A(f',z))$ is uniformly hyperbolic. By combining with the result above, the following theorem holds.

\begin{thm}
Let $T:X\to X$ be a strictly ergodic homeomorphism such that $h(T(x))=h(x)+g$ for some $g\in \mathbb G$ where $h:X\to \mathbb G$ is an onto continuous map and $\mathbb G$ is an infinite compact abelian group. For a generic $f\in C^0(X,\mathbb D),$ we have that $U=\partial \mathbb D\setminus \Sigma$ is dense; that is, the associated CMV operators have a Cantor spectrum.

\end{thm}

\subsection{Jacobi matrices}

Let $\mu$ be a nontrivial probability measure on $\mathbb R$ (not supported on a finite set) with a compact support. Then we may define the $n$-th monic orthogonal polynomial $P_n(x)$ by $P_n\perp x^l$ for all $l=0,1,\cdots , n-1.$ Naturally, the $n$-th orthonormal polynomial is given as $p_n:=P^n/||P^n||.$ It is well known that the orthonormal polynomials obey the Jacobi recursion,
$$xp_n(x) = a_{n+1}p_{n+1}(x) + b_{n+1}p_n(x) + a_np_{n-1}(x)$$
with suitably chosen real valued sequences $a_n>0$ and $b_n,$ called Jacobi coefficients. Conversely, given real valued bounded sequences $\{a_n\}$ and $\{b_n\}$ with $a_n>0$ for all $n\in \mathbb Z_+,$ the Jacobi recursion gives us a set of orthonormal polynomials with respect to a nontrivial probability measure with a compact support.

The Jacobi matrix associated to the measure $\mu$ is the matrix representation for multiplication by $x$ in $L^2( \mathbb  R,d\mu)$ with respect to the basis $\{p_0,p_1,p_2,\cdots\}.$

$$
\mathcal J=\begin{bmatrix}
    & b_0 & a_0& & &  \\
   & a_0 & b_1 & a_1 & &   \\
     &  &a_1 &b_2 & a_2  & &  \\
       
  & &  & \ddots&  \ddots& \ddots\\
\end{bmatrix}
$$

Let $(X,\nu)$ be a probability measure space. Let $f_a,f_b:X\to \mathbb R$ be measurable maps with $f_a(x)>0$ for all $x\in X$ and let $T:X\to X$ be an invertible ergodic transformation. 

As for the case of Verblunsky coefficients, we may consider dynamically defined Jacobi coefficients under this setting. Specifically, two-sided Jacobi coefficients $\{a_n\}_{n\in \mathbb Z}$ and $\{b_n\}_{n\in \mathbb Z}$ are defined by $a_n=f_a(T^n(x))$ and $b_n=f_b(T^nx)$, respectively. Then, an associated bi-infinite Jacobi matrix naturally arises. As in the case of CMV matrices, there are many advantages of bi-infinite Jacobi matrices to study spectral properties. Given $x\in X,$ the bi-infinite Jacobi matrix $H_x$ is given by

$$
H_x=\begin{bmatrix}
    &\ddots&  \ddots& \ddots & & &   &\\
& & a_{-2}& b_{-1} & a_{-1}& & &  \\
 & & &  a_{-1}& b_0 & a_0& & &  \\
   & & & & a_0 & b_1 & a_1 & &   \\
     & & & &  &a_1 &b_2 & a_2  & &  \\     
   & & & & &   & \ddots&  \ddots& \ddots
\end{bmatrix}
$$

It is well known that the spectrum of the Jacobi matrix is closely related to the solutions of the difference equation, $(H_x-E)u=0$ where $E\in \mathbb R.$ Notice that a sequence $\{u_n\}$ is a solution of $(H_x-E)u=0$ if and only if
$$a_nu_{n+1}+(b_n-E)u_n+a_{n-1}u_{n-1}=0$$
for all $n\in \mathbb Z.$

Equivalently, $\{u_n\}$ obeys
$$\begin{bmatrix}
    u_{n}    \\
    a_{n-1}u_{n-1}     
\end{bmatrix}=
A^n_{E,a,b} (x)
\begin{bmatrix}
    u_{0}    \\
    a_{-1}u_{-1}     
\end{bmatrix}
$$
where
$$A_{E,a,b} (x)=\frac 1 {f_a(x)}\begin{bmatrix}
    E-f_b(x)    & -1  \\
    f_a(x)^2     & 0
\end{bmatrix}
$$

Let $X$ be a compact metric space. If we assume that $f_a, f_b\in C^0(X,\mathbb R)$ and $T:X\to X$ is a minimal homeomorphism, $\sigma(H_x)$ coincides for all $x\in X. $

 Let $\Sigma$ be the spectrum of $H_x.$
Marx \cite{MR3291142} shows that  
$$\mathbb R\setminus \Sigma=\{E\in \mathbb R | (T,A_{E,a,b}) \text{ is uniformly hyperbolic}\}.$$

In fact, Marx \cite{MR3291142} considers both singular and non-singular cocycles. For singular cocycles, uniform hyperbolicity is not applicable and, thus, ``a dominated splitting'' is introduced in \cite{MR3291142}. For $SL(2,\mathbb R)$-cocycles, dominated splitting is equivalent to uniform hyperbolicty. (See \cite{MR3291142}.)

Later, given $\epsilon>0,$ we will prove that if $(T,A_{E,a,b})$ is not uniformly hyperbolic, there exists $f_{b'}\in C^0(X,\mathbb R)$ such that $||A_{E,a,b}-A_{E,a,b'}||_{C^0}<\epsilon$ and $(T,A_{E,a,b'})$ is uniformly hyperbolic where $b'_n=f_{b'}(T^nx).$ Together with the result in \cite{MR3291142}, this implies the following theorem. 

\begin{thm}
Let $T:X\to X$ be a strictly ergodic homeomorphism such that $h(T(x))=h(x)+g$ for some $g\in \mathbb G$ where $h:X\to \mathbb G$ is an onto continuous map and $\mathbb G$ is an infinite compact abelian group. Let $f_a\in C^0(X,\mathbb R)$ with $f_a(x)>0$ for all $x\in X$. For generic $f_b\in C^0(X,\mathbb R),$ we have that $\mathbb R\setminus \Sigma$ is dense; that is, the associated Jacobi matrices have Cantor spectrum.
\label{Jacobi_thm}
\end{thm}

\subsection{Discussion of the results}

In addition to the results for general continuous $SL(2,\mathbb R)$-cocycles, the $C^0$-genericity of Cantor spectrum for Schr{\" o}dinger operators proved in Avila \emph{et al} \cite{MR2477761} was striking. Especially for the standard skew-shift with a sufficiently regular nonconstant potential function $V:\mathbb T^2\to \mathbb R$, it had been widely expected to have pure point spectrum, which is not a Cantor set, with exponentially decaying eigenfunctions.

After the work, an obvious expectation for CMV and Jacobi matrices generated by the same base dynamics is to also have generic Cantor spectrum. The present paper, to the best of our knowledge, first provides a proof for the statement.

We would like to mention that the spectrum of quasiperiodic Schr{\"o}dinger operators with analytic potentials behaves in a very different way. For the case of shifts on the one-dimensional torus, Goldstein and Schlag \cite{MR2753606} proved that Cantor spectrum is obtained for analytic potentials in the regime of positive Lyapunov exponents with typical shifts, i.e., with $x\in \mathbb T$ and $n\in \mathbb Z,$ shifts $x+n\alpha$ for Lebesgue almost every $\alpha\in \mathbb T$. For the case of shifts on a multidimensional torus, it turned out to be harder to study and, thus, is much less understood. However, for a two-dimensional shift, Goldstein \emph{et al} \cite{MR3987178} show that the spectrum consists of a single interval for large real analytic potentials satisfying certain restrictions.

\section{Results for $SL(2,\mathbb R)$-cocycles}

As noted in the introduction, our work is closely related to the results in Avila \emph{et al}  \cite{MR2477761}. In this section, we discuss the results in \cite{MR2477761} for general continuous $SL(2,\mathbb R)$-cocycles over the same base dynamics as in the present paper, i.e., with a strictly ergodic homeomorphism $T:X\to X$ such that $h(T(x))=h(x)+g$ for some $g\in \mathbb G$ where $h:X\to \mathbb G$ is an onto continuous map and $\mathbb G$ is an infinite compact abelian group. 

We say that two cocycles $(T,A)$ and $(T,\tilde A)$ are conjugate (respectively, $PSL(2,\mathbb R)$- conjugate) if there exists a conjugacy $B\in C^0(X,SL(2,\mathbb R))$  
 $(\text{respectively, } B\in C^0(X, PSL(2, \mathbb R)))$ such that $ \tilde A(x)=B(T(x))A(x)B(x)^{-1}.$

We say $(T,A)$ is reducible if it is $PSL(2,\mathbb R)-$conjugate to a constant cocycle. We say $(T,A)$ is reducible up to homotopy if there exists a reducible cocycle $(T, \tilde A)$ such that the maps $A$ and $\tilde A:X\to SL(2,\mathbb R)$ are homotopic. Let \emph{Ruth} be the set of all $A$ such that $(T,A)$ is reducible up to homotopy.

In Avila \emph{et al}  \cite{MR2477761}, the authors show that if an $SL(2,\mathbb R)$-cocycle is not uniformly hyperbolic, it can be approximated by one that is conjugate to an $SO(2,\mathbb R)$-cocycle. Using this, it is proved that if a cocycle is in \emph{Ruth,} then it can be approximated by a uniformly hyperbolic cocycle. As a uniformly hyperbolic cocycle is always reducible up to homotopy, this shows that uniform hyperbolicity is dense in \emph{Ruth}. (\cite{MR2477761}, Theorem 2.)
 
 We will observe that a cocycle associated to a CMV matrix or a Jacobi matrix is homotopic to a constant cocycle, hence in \emph{Ruth.} Therefore, it can be $C^0$-perturbed so that it is a continuous $SL(2,\mathbb R)$-cocycle, which is uniformly hyperbolic. For the purpose of our work, a difficulty is that the perturbed cocycle need not be in the form of one associated with CMV matrices or Jacobi matrices. 
 
 The difficulty is nicely overcome for the case of Schr{\"o}dinger operators in Avila \emph{et al}  \cite{MR2477761} by using ``a projection lemma'', which is also proved in their work. On the one hand, it makes the perturbed $SL(2,\mathbb R)$-cocycle conjugate to a cocycle associated to Schr{\"o}dinger operators (one may say the perturbed $SL(2,\mathbb R)$-cocycle is projected). Of course, the conjugacy preserves the uniform hyperbolicity.  On the other hand, the associated cocyle can be arbitrarily close to the original cocycle, which is not uniformly hyperbolic. In conclusion, it provides a uniformly hyperbolic cocyle associated to Schr{\"o}dinger operators such that it is arbitrarily close to the original cocycle, which was not uniformly hyperbolic.  
 
 Our proof for the case of Jacobi matrices is a direct application of the above procedure. We were unable to find a possible way of application for the case of CMV matrices and we will explain in more detail this difficulty during the proof for the case of Jacobi matrices. However, we still use half of the projection lemma in \cite{MR2477761}. We first need to introduce some notation. Given $A\in C^0(X,SL(2,\mathbb R))$ and a nonempty subset $V\subset X,$ let $C^0_{A,V}(X,SL(2,\mathbb R))\subset C^0(X,SL(2,\mathbb R))$ be the set of all $B\in C^0(X,SL(2,\mathbb R))$ such that $B(x)=A(x)$ for $x\notin V.$

\begin{lem}(\cite{MR2477761}, Lemma 10)

Let $V\subset X$ be any nonempty open set, and let $A\in C^0(X, SL(2,\mathbb R)$ be arbitrary. Then there exist an open neighborhood $\mathcal W_{A,V}\subset C^0(X,SL(2,\mathbb R))$ of $A$ and continuous maps 
$$\Phi=\Phi_{A,V}:\mathcal W_{A,V}\to C^0_{A,\overline V}(X,SL(2,\mathbb R))$$
and 
$$\Psi=\Psi_{A,V}:\mathcal W_{A,V}\to C^0(X,SL(2,\mathbb R))$$
satisfying 
$$\Psi(B)(T(x))\cdot B(x)\cdot [\Psi(B)(x)]^{-1}=\Phi(B)(x),$$
$$\Phi(A)=A\text{ and } \Psi(A)=id.$$
\label{ABD_lemma10}
\end{lem}

\section{Proof of results}

\subsection{Proof for CMV matrices}

Let $J=\begin{bmatrix}
    1    & 0  \\
   0   & -1
\end{bmatrix}.$ Recall $ U (1,1)$ is the group of all $2\times 2$ matrices obeying
$$A^* J A=J$$
and ${SU}(1,1):=\{A\in U(1,1):\det A=1\}.$

Let $\overline J$ be a matrix such that $\overline J^*=\overline J=\overline J^{-1}$ and $Tr(\overline J)=0.$ Then, we may choose a unitary matrix $W$ such that $W\overline J W^{-1}=J.$ Thus, if we define ${SU}(1,1;\overline J):=\{A: A^*\overline J A=\overline J\}$, we have
$$W {SU}(1,1;\overline J)W^{-1}={SU}(1,1).$$

We introduce an important result here:

\begin{lem} (\cite{MR2105089}, Proposition 10.4.1)

With $J_r:=\begin{bmatrix}
    0    & i   \\
    -i     & 0
\end{bmatrix},$ we have
$$ {SU}(1,1, J_r)=SL(2,\mathbb R).$$
\end{lem}

For our purpose, this may be read as
$$W^{-1} {SU}(1,1)W= {SU}(1,1, J_r)=SL(2,\mathbb R)$$
where $W=\frac{1}{\sqrt 2}\begin{bmatrix}
    1    & i   \\
    1     & -i
\end{bmatrix}.$
Let $f\in C^0(X,\mathbb D)$ and let $z\in \partial \mathbb D$ be given. Note that $\overline A(f,z)\in C^0(X,{SU}(1,1)).$ Thus, given $x\in X,$ we have an $SL(2,\mathbb R)$ matrix,

\begin{align*}
W^{-1}&\overline A(f(x),z)W\\
&=\frac{1}{2z^{1/2} \sqrt{1-|f(x)|^2} }
\begin{bmatrix}
    {z-\bar f(x)-f(x)z+1}   & {i(z+\bar f(x)- f(x)z-1)}   \\
    {i(-z+\bar f (x)-f(x)z+1)}       & {z+\bar f(x)+ f(x)z+1}
\end{bmatrix}.
\end{align*}

Let $z=e^{i\psi}$ and let $f(x)=r(x)e^{i\phi(x)}.$ Then, we have

\begin{equation*}\label{theta-alt}
\begin{aligned}
W^{-1}&\overline A(f(x),z)W=\frac{1}{2\sqrt{1-r(x)^2} } \times\\
&
\left[\begin{matrix}
  \exp(\frac{i\psi}{2})-r(x)\exp(-\frac{i\psi}{2}-i\phi(x))-r(x)\exp(\frac{i\psi}{2}+i\phi(x))+\exp(-\frac{i\psi}{2})\\
  i\Big(-\exp(\frac{i\psi}{2})+r(x)\exp(-\frac{i\psi}{2}-i\phi(x))-r(x)\exp(\frac{i\psi}{2}+i\phi(x))+\exp(-\frac{i\psi}{2})\Big)
\end{matrix}\right.\\
&\qquad\qquad
\left.\begin{matrix}
  i\Big(\exp(\frac{i\psi}{2})+r(x)\exp(-\frac{i\psi}{2}-i\phi(x))-r(x)\exp(\frac{i\psi}{2}+i\phi(x))-\exp(-\frac{i\psi}{2})\Big)\\
  \exp(\frac{i\psi}{2})+r(x)\exp(-\frac{i\psi}{2}-i\phi(x))+r(x)\exp(\frac{i\psi}{2}+i\phi(x))+\exp(-\frac{i\psi}{2})
\end{matrix}\right].
\end{aligned}
\end{equation*}

Thus, by a simple observation, we have following.

\begin{lem}
Let $z=e^{i\psi}\in \partial \mathbb D$ and let $f\in C^0(X,\mathbb D)$ be given by $f(x)=r(x)e^{i\phi(x)}.$ Then $W^{-1}\overline A(f(x),z)W$ is equal to
$$\frac{1}{\sqrt{1-r(x)^2}}\Big(\begin{bmatrix}
    \cos{\theta'}  & -\sin {\theta'}  \\
    \sin {\theta'}      & \cos {\theta'}
\end{bmatrix}
+ r(x)
\begin{bmatrix}
    -\cos\theta (x)   & \sin \theta(x)  \\
    \sin \theta  (x)     & \cos \theta(x)
\end{bmatrix}\Big)$$
where $\theta'=\frac \psi 2$ and $\theta(x)=\frac {\psi}2+\phi(x).$
\label{A}
\end{lem}

Given $z\in e^{i\psi}\in \partial \mathbb D$ with $\theta':=\psi/2,$ define $S'\subset SL(2,\mathbb R)$ as 
\begin{align*}
S'=\Big\{
\frac{1}{\sqrt{(1-s^2)}}\Big(\begin{bmatrix}
    \cos\theta'   & -\sin\theta'\\
    \sin \theta'    & \cos\theta'
\end{bmatrix}
+ s
&\begin{bmatrix}
    -\cos\theta    & \sin \theta  \\
    \sin \theta       & \cos \theta
\end{bmatrix}\Big) \\
&:s\in [0,1), \theta\in \mathbb R \Big\}.
\end{align*}
Let $A(x):=W^{-1}\overline A(f(x),z)W$. Then, we have $A \in C^0(X,S')$ and we may write it as in Lemma \ref{A}.

We assume that $f\in C^0(X,\mathbb D)$ is not identically zero. Let $y\in X$ be an element such that $f(y)\neq 0.$ We may choose a nonempty open set $V\subset X$ so that $y\in V$ and $f(x)\neq 0$ for all $x\in \overline V.$ In fact, there exist $r_1,r_2\in [0,1)$ such that $r_1\leq |f(x)|\leq r_2$ for all $x\in \overline V.$ Thus, $r_1\leq r(x)\leq r_2$ for all $x\in \overline V$.

By Lemma \ref{ABD_lemma10}, we may choose a uniformly hyperbolic cocycle $(T,B)$ such that $B\in C^0_{A,\bar V}(X,SL(2,\mathbb R))$ is arbitrarily $C^0-$ close to $A$. Write $B$ as
\begin{equation*}
B(x)=\frac{1}{\sqrt{1-r(x)^2} }\Big(\begin{bmatrix}
    \cos \theta'    & -\sin\theta'   \\
    \sin\theta'      & \cos \theta'
\end{bmatrix}
+ r(x)
\begin{bmatrix}
    b_{11}(x)      & b_{12}(x)   \\
    b_{21}(x)      & b_{22}(x)   
\end{bmatrix}\Big).
\end{equation*}

Set $B'(x):=B(x)$ for $x\in X\setminus \overline V.$ For $x\in \overline V,$ we define $B'\in C^0(X,S')$ as follows:

\begin{enumerate}[label=(\Roman*)]
\item Denote $R_\eta$ as the $2\times 2$ rotation matrix with the angle $\eta\in \mathbb R.$ Recall since $(T,B)$ is uniformly hyperbolic there exists a continuous map $u$ from $X$ to the unit circle in $\mathbb R^2$ such that $u(x)\in E^u_x.$ Let $R_{-\tau(x)}\cdot u(x)=(1,0).$ Consider the matrix
\begin{equation}
\begin{bmatrix}
    y_{11}(x)      & y_{12}(x)   \\
    y_{21}(x)      & y_{22}(x)   
\end{bmatrix}
:=
\begin{bmatrix}
    b_{11}(x)      & b_{12}(x)   \\
    b_{21}(x)      & b_{22}(x)   
\end{bmatrix}
\cdot 
R_{\tau(x)}.
\label{Y}
\end{equation}

\item Normalize the vector $(y_{11}(x), y_{21}(x))$ (then, we may write it as $(-\cos \tilde \theta, \sin \tilde \theta)$ for some $\tilde \theta\in \mathbb R)$ and replace the vector $(y_{12}(x), y_{22}(x))$ by $(\sin \tilde \theta, \cos \tilde \theta).$ 
\item Set $B'$ as
\begin{align*}
B'(x)=\frac{1}{\sqrt{1-r(x)^2} }\Big(&\begin{bmatrix}
    \cos \theta'    & -\sin\theta'   \\
    \sin\theta'      & \cos \theta'
\end{bmatrix}\\
&+ r(x)
\begin{bmatrix}
    -\cos\tilde \theta(x)      & \sin\tilde \theta (x)  \\
    \sin\tilde \theta (x)     & \cos\tilde \theta (x)  
\end{bmatrix}
R_{-\tau(x)} \Big).
\end{align*}

\end{enumerate}

Observe that, as an element in the projective line of $\mathbb R^2,$ we have
\begin{align*}
\begin{bmatrix}
    -\cos\tilde \theta(x)      & \sin\tilde \theta (x)  \\
    \sin\tilde \theta (x)     & \cos\tilde \theta (x)  
\end{bmatrix}
R_{-\tau(x)}
\cdot u(x)
&=
\begin{bmatrix}
    -\cos\tilde \theta(x)      & \sin\tilde \theta (x)  \\
    \sin\tilde \theta (x)     & \cos\tilde \theta (x)  
\end{bmatrix}
\begin{bmatrix}
    1         \\
    0
\end{bmatrix}\\
&=
\begin{bmatrix}
    y_{11}(x)      & y_{12}(x)   \\
    y_{21}(x)      & y_{22}(x)   
\end{bmatrix}
\begin{bmatrix}
    1         \\
    0
\end{bmatrix}\\
&=
\begin{bmatrix}
    b_{11}(x)      & b_{12}(x)   \\
    b_{21}(x)      & b_{22}(x)   
\end{bmatrix}
u(x).
\end{align*}

\begin{lem}
Given $\epsilon>0,$ there exists $\delta>0$ so that $||A-B||_{C^0}<\delta$ implies $||B-B'||_{C^0}<\epsilon.$
\end{lem}

\begin{proof}

It suffices to show that given $\epsilon >0 $, there exists $\delta>0$ so that 
$$\Big| \Big|\begin{bmatrix}
    b_{11}(x)      & b_{12}(x)   \\
    b_{21}(x)      & b_{22}(x)   
\end{bmatrix}
-
\begin{bmatrix}
    -\cos\theta (x)   & \sin \theta(x)  \\
    \sin \theta  (x)     & \cos \theta(x)
\end{bmatrix}\Big| \Big|_{C^0}<\delta$$
implies 

$$\Big| \Big|\begin{bmatrix}
    -\cos\tilde \theta(x)      & \sin\tilde \theta (x)  \\
    \sin\tilde \theta (x)     & \cos\tilde \theta (x)  
\end{bmatrix}R_{-\tau(x)}
-
\begin{bmatrix}
    b_{11}(x)  &b_{12}(x)        \\
    b_{21} (x)&b_{22}(x)     
\end{bmatrix}\Big| \Big|_{C^0}<\epsilon.
$$

Set $\delta=\epsilon/4.$ Then, 

$$\Big| \Big|\begin{bmatrix}
    y_{11}(x)      & y_{12}(x)   \\
    y_{21}(x)      & y_{22}(x)   
\end{bmatrix}
-
\begin{bmatrix}
    -\cos\theta (x)   & \sin \theta(x)  \\
    \sin \theta  (x)     & \cos \theta(x)
\end{bmatrix}R_{\tau(x)}\Big| \Big|_{C^0}<\epsilon/4.$$

In particular, we have

$$\Big| \Big|\begin{bmatrix}
    y_{11}(x)        \\
    y_{21}(x)     
\end{bmatrix}
-
\begin{bmatrix}
    -\cos(\theta (x)+\tau(x))    \\
    \sin (\theta (x)+\tau(x))
\end{bmatrix}\Big| \Big|<\epsilon/4
$$
and
$$\Big| \Big|\begin{bmatrix}
    y_{12}(x)        \\
    y_{22}(x)     
\end{bmatrix}
-
\begin{bmatrix}
    \sin(\theta (x)+\tau(x))    \\
    \cos (\theta (x)+\tau(x))
\end{bmatrix}\Big| \Big|<\epsilon/4.
$$

This implies that 
$$\Big| \Big|\begin{bmatrix}
    y_{21}(x)        \\
    -y_{11}(x)     
\end{bmatrix}
-
\begin{bmatrix}
    y_{12}(x)        \\
    y_{22}(x)     
\end{bmatrix}\Big| \Big|<\epsilon/2.
$$
Therefore, we have
$$\Big| \Big|\begin{bmatrix}
    y_{11}(x)  &y_{21}(x)        \\
    y_{21}(x) &-y_{11}(x)     
\end{bmatrix}
-
\begin{bmatrix}
     y_{11} (x) &y_{12}(x)        \\
     y_{21} (x) &y_{22}(x)     
\end{bmatrix}\Big| \Big|_{C^0}<\epsilon/2.
$$

If necessary, choose a $C^0-$ closer $B$ to $A$ (so, smaller $\delta$) so that 
$$\Big|\Big|\frac 1{\sqrt{y_{11}(x)^2+y_{21}(x)^2}}\begin{bmatrix}
    y_{11} (x) &y_{21}(x)        \\
    y_{21} (x)&-y_{11}(x)     
\end{bmatrix}
-\begin{bmatrix}
    y_{11}(x)  &y_{21}(x)        \\
    y_{21} (x)&-y_{11}(x)     
\end{bmatrix}\Big|\Big|_{C^0}<\epsilon /2$$
for all $x\in X.$

Then, by the triangle inequality,
$$\Big| \Big|\frac{1}{\sqrt{y_{11}(x)^2+y_{21}(x)^2  }}
\begin{bmatrix}
    y_{11} (x) &y_{21}(x)        \\
    y_{21} (x)&-y_{11}(x)     
\end{bmatrix}
-
\begin{bmatrix}
    y_{11}(x)  &y_{12}(x)        \\
    y_{21}(x) &y_{22}(x)     
\end{bmatrix}\Big| \Big|_{C^0}<\epsilon.
$$

In conclusion,

$$\Big| \Big|\begin{bmatrix}
    -\cos\tilde \theta(x)      & \sin\tilde \theta (x)  \\
    \sin\tilde \theta (x)     & \cos\tilde \theta (x)  
\end{bmatrix}
-
\begin{bmatrix}
    y_{11} (x) &y_{12}(x)        \\
    y_{21} (x)&y_{22}(x)     
\end{bmatrix}\Big| \Big|_{C^0}<\epsilon,
$$

which implies

$$\Big| \Big|\begin{bmatrix}
    -\cos\tilde \theta(x)      & \sin\tilde \theta (x)  \\
    \sin\tilde \theta (x)     & \cos\tilde \theta (x)  
\end{bmatrix}R_{-\tau(x)}
-
\begin{bmatrix}
    b_{11}(x)  &b_{12}(x)        \\
    b_{21} (x)&b_{22}(x)     
\end{bmatrix}\Big| \Big|_{C^0}<\epsilon
.$$

\end{proof}

Let $N$ be the closed annulus on the $\mathbb R^2$-plane centered at the origin with radius $r_1/\sqrt{1- r_1^2}\leq \rho \leq r_2/\sqrt{1- r_2^2}$. That is,
$$N:=\{\frac{r}{\sqrt{1-r^2}}(\cos\eta, \sin \eta )\in \mathbb R^2: r_1 \leq r \leq r_2, \eta\in \mathbb R \}.$$

 Note that 
$$\frac{1}{\sqrt{1-r^2}}-\frac{r}{\sqrt{1-r^2}}<1.$$
Let $\underline \epsilon<1$ be a number such that 
$$\frac{1}{\sqrt{1-r_1^2}}-\frac{\underline \epsilon \cdot  r_1}{\sqrt{1- r_1^2}}<1$$

and let $\overline \epsilon>1$ be a number such that 
$$\frac{1}{\sqrt{1- r_2^2}}-\frac{\overline \epsilon \cdot   r_2}{\sqrt{1-  r_2^2}}>0.$$

Given $ \underline \epsilon <\epsilon <\overline \epsilon,$ we define $h_{\epsilon}:N \to [0,1)$ and $g_{\epsilon}:N \to [-\pi,\pi]$ as follows:

 Given $t\in N$ with $|t|=r/\sqrt{1-r^2}.$ we may choose $\eta\in \mathbb R$ such that 
$$t=\frac{r}{\sqrt{1-r^2}}
\begin{bmatrix}
    -\cos\eta    & \sin\eta \\
    \sin \eta    & \cos \eta 
\end{bmatrix}\cdot
\begin{bmatrix}
    1  \\
    0  
\end{bmatrix}
$$
Consider the vector 
\begin{equation}
\frac{1}{\sqrt{1-r^2}}\Big(\begin{bmatrix}
-1\\
0
\end{bmatrix}
+ r
\begin{bmatrix}
    -\cos\eta    & \sin\eta \\
    \sin \eta    & \cos \eta 
\end{bmatrix}\cdot
\begin{bmatrix}
   \epsilon\\
    0
\end{bmatrix}\Big).
\label{F}
\end{equation}

Then, there exist unique $s\in [0,1)$ and $\beta\in [-\pi,\pi]$ so that 
$$
\frac{1}{\sqrt{1-s^2}}\Big(\begin{bmatrix}
-1\\
0
\end{bmatrix}
+ s
\begin{bmatrix}
    -\cos\eta    & \sin\eta \\
    \sin \eta    & \cos \eta 
\end{bmatrix}\cdot R_{\beta}
\cdot
\begin{bmatrix}
  1\\
    0
\end{bmatrix}\Big)$$

coincides with (\ref{F}).
We define $h_\epsilon (t)=s$ and $g_\epsilon (t)=\beta.$
Here are some properties of $h_\epsilon:$

\begin{lem}
\begin{enumerate}[label=(\alph*)]
\item $h_\epsilon$ is continuous.
\item given $t\in N$ with $|t|=r/\sqrt{1-r^2},$ we have $h_\epsilon(t)\to r$ as $\epsilon \to 1.$
\item let $t\in N$ with $|t|=r/\sqrt{1-r^2}$ and let $\epsilon'\leq \epsilon\leq1.$ Then, we have $h_{\epsilon'}(t)\leq h_{\epsilon}(t)\leq r.$
\item let $t\in N$ with $|t|=r/\sqrt{1-r^2}$ and let $\epsilon'\geq \epsilon\geq 1.$ Then, we have $h_{\epsilon'}(t)\geq h_{\epsilon}(t)\geq r.$
\item let $h:N\to [0,1)$ be the function defined by $h(t)=r$ if $|t|=r/\sqrt{1-r^2}.$ Let $\{\epsilon_n\}$ be a sequence such that  $\epsilon_n\to 1$ and $\epsilon_n\leq \epsilon_{n+1}$ (or $\epsilon_n\geq \epsilon_{n+1}$) for all $n.$ Then, $\{h_{\epsilon_n}\}$converges uniformly to $h.$
 \end{enumerate}
 \label{h_lemma}
 \end{lem}
 
 \begin{proof}
Parts (a) to (d) are easy to check. For part (e), suppose that $\{\epsilon_n\}$ is a sequence such that  $\epsilon_n\to 1$ and $\epsilon_{n}\leq \epsilon_{n+1}\leq 1$ for all $n.$ By part (b), $\{h_{\epsilon_n}\} $ pointwise converges to $h$. By part (c), $h_{\epsilon_n}(t)\leq h_{\epsilon_{n+1}}(t)$ for all $n$ and for all $t\in T.$ 
 
 Thus, by Dini's theorem, $\{h_{\epsilon_n}\}$ uniformly converges to $h.$ A similar argument shows the uniform convergence of $\{h_{\epsilon_n}\}$ with $\epsilon_{n}\geq \epsilon_{n+1}\geq 1$ for all $n.$ 
  \end{proof}

Here are some properties of $g_\epsilon:$

\begin{lem}
\begin{enumerate}[label=(\alph*)]
\item $g_\epsilon$ is continous.
\item given $t\in N,$ we have $g_\epsilon(t)\to 0$ as $\epsilon \to 1.$
\item let $t=\rho (\cos \eta, \sin \eta)\in N$ with $\eta\in [0,\pi].$ If $\epsilon'\leq \epsilon\leq 1,$ we have $g_{\epsilon'}(t)\leq g_{\epsilon}(t)\leq 0$. If $\epsilon'\geq \epsilon\geq 1,$ we have $g_{\epsilon'}(t)\geq g_{\epsilon}(t)\geq 0$.
\item let $t=\rho (\cos \eta, \sin \eta)\in N$ with $\eta\in [\pi,2\pi].$ If $\epsilon'\leq \epsilon \leq 1,$ we have $g_{\epsilon'}(t)\geq g_{\epsilon}(t)\geq 0$. If $\epsilon'\geq \epsilon\geq 1,$ we have $g_{\epsilon'}(t)\leq g_{\epsilon}(t)\leq 0$.
 \item let $\{\epsilon_n\}$ be a sequence such that  $\epsilon_n\to 1$ and $\epsilon_n\leq \epsilon_{n+1}$ (or $\epsilon_n\geq \epsilon_{n+1}$) for all $n.$  Then, $\{g_{\epsilon_n}\}$ uniformly converges to $g=0.$
 \end{enumerate}
 \label{g_lemma}
 \end{lem}
 
 \begin{proof}
Parts (a) to (d) are easy to check. For part (e), suppose that $\{\epsilon_n\}$ is a sequence such that  $\epsilon_n\to 1$ and $\epsilon_{n}\leq \epsilon_{n+1}\leq 1$ for all $n.$ By part (b), $\{g_{\epsilon_n}\} $ converges pointwise to $g=0$. By part (c), $g_{\epsilon_n}(t)\leq g_{\epsilon_{n+1}}(t)$ for all $n$ if $t= \rho (\cos \eta, \sin \eta)$ with $\eta\in [0,\pi]$. By part (d), $g_{\epsilon_n}(t)\geq g_{\epsilon_{n+1}}(t)$ for all $n$ if $t= \rho (\cos \eta, \sin \eta)$ with $\eta\in [\pi, 2\pi].$

Thus, by Dini's theorem, $\{g_{\epsilon_n}\}$ uniformly converges to $g=0.$ A similar argument shows the uniform convergence of $\{g_{\epsilon_n}\}$ with $\epsilon_{n}\geq \epsilon_{n+1}\geq 1$ for all $n.$ 
 \end{proof}

Note that by setting $h(\epsilon, t):=h_\epsilon(t)$ and $g(\epsilon,t)=g_\epsilon(t)$, it is easy to see that $h:(\underline \epsilon, \overline \epsilon)\times N\to \mathbb [0,1)$ and $g:(\underline \epsilon, \overline \epsilon)\times N\to \mathbb [-\pi,\pi]$ are continuous functions. 

Now, set $B''(x)=B(x)$ for $x\in X\setminus \overline V.$ For $x\in \overline V,$ define $B''\in C^0(X,S')$ as follows:

\begin{enumerate}[label=(\Roman*)]
\item Consider 
\begin{align*}B'(x)=\frac{1}{\sqrt{1-r(x)^2} }\Big(&\begin{bmatrix}
    \cos \theta'    & -\sin\theta'   \\
    \sin\theta'      & \cos \theta'
\end{bmatrix}\\
&+ r(x)
\begin{bmatrix}
    -\cos\tilde \theta(x)      & \sin\tilde \theta (x)  \\
    \sin\tilde \theta (x)     & \cos\tilde \theta (x)  
\end{bmatrix}
R_{-\tau(x)} \Big)
\end{align*}
and $\epsilon(x)=\sqrt{y_{11}(x)^2+y_{21}(x)^2 }$ (see (\ref{Y})). We may assume $\underline \epsilon <\epsilon(x)<\overline \epsilon$ for all $x\in X$ by taking a $C^0-$close enough $B$ to $A.$

\item Let $\omega(x)\in [0,2\pi]$ be such that $R_{\omega(x)}\cdot R_{\theta'}\cdot u(x)=(-1,0).$ Define $t(x)\in N$ as 
$$t(x)=
\rho(x) \cdot R_{\omega(x)}\cdot\begin{bmatrix}
    -\cos\tilde \theta(x)      & \sin\tilde \theta (x)  \\
    \sin\tilde \theta (x)     & \cos\tilde \theta (x)  
\end{bmatrix}R_{-\tau(x)}
u(x)
$$
where $\rho(x)=r(x)/\sqrt{1-r(x)^2}.$

\item Define $B''(x)$ as 
\begin{align*}
B''(x)=\frac{1}{\sqrt{1-s^2} }\Big(&\begin{bmatrix}
    \cos \theta'    & -\sin\theta'   \\
    \sin\theta'      & \cos \theta'
\end{bmatrix}\\
&+ s
\begin{bmatrix}
    -\cos\tilde \theta(x)      & \sin\tilde \theta (x)  \\
    \sin\tilde \theta (x)     & \cos\tilde \theta (x)  
\end{bmatrix}
R_{-\tau(x)}R_{\beta} \Big)
\end{align*}
where  $s=h_{\epsilon(x)}(t(x))$ and let $\beta=g_{\epsilon(x)}(t(x)).$

\end{enumerate}

Note that, by construction, we have $B''\in C^0(X,S')$ and
\begin{equation}
B(x)u(x)=B''(x)u(x)
\label{eq1}
\end{equation}
for all $x\in X.$

\begin{lem} The cocycle $(T,B'')$ is uniformly hyperbolic.
\end{lem}
\begin{proof}

Recall that $B\in C^0_{A,\overline V}(X,SL(2,\mathbb R))$ is chosen so that $(T,B)$ is uniformly hyperbolic. By the condition (2) in the definition \ref{UH}, we have
$$||B^{-n}(T^n(x))u(T^n(x))||\leq C\lambda^n$$
for all $x\in X$ and $n\geq 1.$

Recall $B^k(x)$ where $ k\in \mathbb Z$ is defined by $(T,B)^k=(T^k, B^k).$ Therefore, we have
$$B^{-n}(T^n(x))=B(x)^{-1}\cdots B(T^{n-1}(x))^{-1}$$
and
$$B^n(x)=B(T^{n-1})\cdots B(x).$$
In other word, 
\begin{equation}
[B^n(x)]^{-1}=B^{-n}(T^n(x)).
\label{eq2}
\end{equation}
Thus,
$$|| B^n(x) B^{-n}(T^n(x)) u(T^n(x)   ||=|| u(T^n(x))||=1.$$

On the other hand, by the condition (1) in the definition \ref{UH}, we have
$$B^n(x)E^u_x=E^u_{T^n(x)},$$
which is equivalent to 
$$B^{-n}(T^n(x))E^u_{T^n(x)}=E^u_{x}$$
by the equation (\ref{eq2}).

This implies
$$B^{-n}(T^n(x))u(T^n(x))=\pm ||  B^{-n}(T^n(x))u(T^n(x))   ||u(x).$$
Thus, by applying $B^n(x)$ on both sides and taking norm,
\begin{align*}
|| B^n(x) B^{-n}(T^n(x)) u(T^n(x)   ||&=     ||  B^{-n}(T^n(x))u(T^n(x))   ||\cdot  || B^n(x)  u(x)   ||\\
&\leq C\lambda^n  || B^n(x)  u(x)   ||
\end{align*}
for all $x\in X$ and $n\geq 1$ where the last inequality holds by the condition (2) of the definition \ref{UH}.

Together with $|| B^n(x) B^{-n}(T^n(x)) u(T^n(x)   ||=1,$ this implies 
$$ || B^n(x)  u(x)   || \geq C^{-1}\lambda^{-n}$$
for all $x\in X$ and $n\geq 1.$

Therefore, 
$$|| (B'')^n(x)   ||\geq  || (B'')^n(x)  u(x)   || =|| B^n(x)  u(x)   || \geq C^{-1}\lambda^{-n}$$
for all $x\in X$ and $n\geq 1$ where the equality holds by the equation (\ref{eq1}).
This implies that $(T,B'')$ is uniformly hyperbolic.

\end{proof}

\begin{prop}
Let $f\in C^0(X,\mathbb D)$ with $f(x)\neq 0$ for some $x\in X$ and let $A=W^{-1}\overline A(f, z)W\in C^0(X,S')$. Suppose that $(T,A)$ is not uniformly hyperbolic. Given $\epsilon'>0$, there exists $B''\in C^0(X,S')$ such that 
$||A-B''||_{C^0}<\epsilon'$ and $(T,B'')$ is uniformly hyperbolic.
\label{CMV_prop1}
\end{prop}
\begin{proof}

Write $A$ as in Lemma \ref{A}. Let $\epsilon'>0$ be given. Let $y\in X$ be such that $r(y):=r\neq 0.$ Choose an open set $V\subset X$ such that $y\in V$ and $r(x)\neq 0$ for all $x\in \overline V.$ We may assume that $r_1\leq r(x)\leq r_2$ for all $x\in \overline V$ where $r_1,r_2\in [0,1).$

Choose $\delta_1,\delta_2>0$ so that $|s- r_2|<\delta_1$ and $|\beta|<\delta_2$ imply

\begin{align*}
\Big|\Big|&\frac{1}{\sqrt{1-r_2^2} }\begin{bmatrix}
    \cos \theta'    & -\sin\theta'   \\
    \sin\theta'      & \cos \theta'
\end{bmatrix}
-
\frac{1}{\sqrt{1-s^2} }\begin{bmatrix}
    \cos \theta'    & -\sin\theta'   \\
    \sin\theta'      & \cos \theta'
\end{bmatrix}\Big|\Big|\\
&+  \Big|\Big|\frac{r_2}{\sqrt{1-r^2_2}}
\begin{bmatrix}
    -\cos \eta      & \sin \eta   \\
    \sin \eta    & \cos\eta 
\end{bmatrix}
 -
\frac{s}{\sqrt{1-s^2} }\begin{bmatrix}
    -\cos \eta      & \sin \eta   \\
    \sin \eta    & \cos\eta 
\end{bmatrix}
\Big|\Big|\\
&+  \Big|\Big|\frac{s}{\sqrt{1-s^2} }
\begin{bmatrix}
    -\cos \eta      & \sin \eta   \\
    \sin \eta    & \cos\eta 
\end{bmatrix}
 -
\frac{s}{\sqrt{1-s^2} }\begin{bmatrix}
    -\cos \eta      & \sin \eta   \\
    \sin \eta    & \cos\eta 
\end{bmatrix}R_\beta
\Big|\Big|<\frac{\epsilon'}3
\end{align*}
for all $\eta\in \mathbb R.$

Then, given $r\in [r_1,r_2]$, $\bar s$ with $|r-\bar s|<\delta_1$ and $|\beta|<\delta_2.$ we have
\begin{align*}
\Big|\Big|&\frac{1}{\sqrt{1-r^2} }\Big(\begin{bmatrix}
    \cos \theta'    & -\sin\theta'   \\
    \sin\theta'      & \cos \theta'
\end{bmatrix}
+  r
\begin{bmatrix}
    -\cos \eta      & \sin \eta   \\
    \sin \eta    & \cos\eta 
\end{bmatrix}
 \Big)\\
&- \frac{1}{\sqrt{1-\bar s^2} }\Big(\begin{bmatrix}
    \cos \theta'    & -\sin\theta'   \\
    \sin\theta'      & \cos \theta'
\end{bmatrix}
+ \bar s
\begin{bmatrix}
    -\cos \eta      & \sin \eta   \\
    \sin \eta    & \cos\eta 
\end{bmatrix}
R_\beta\Big)
\Big|\Big|<\frac{\epsilon'}3
\end{align*}
for all $\eta\in \mathbb R.$

Let $N\subset \mathbb R^2$ be the annulus of radius $r_1/\sqrt{1-r_1^2}\leq \rho \leq r_2/\sqrt{1-r_2^2}.$ Define $l:N\to [0,1)$ as $l(t):=r$ when $|t|=r/\sqrt{1-r^2}$ and define $h_\epsilon:N\to [0,1)$ and $g_\epsilon:N\to [-\pi, \pi]$ as above. By parts (c),(d) and (e) of Lemmas \ref{h_lemma} and \ref{g_lemma}, we may choose $\delta>0$ so that $|1-\epsilon|<\delta$ implies $|h_\epsilon (t)-l(t)|<\delta_1$ and $|g_\epsilon(t)|<\delta_2$ for all $t\in N$.

By Lemma \ref{ABD_lemma10}, we may choose $B\in C^0_{A,\overline V}(X,SL(2,\mathbb R)),$ which is arbitrarily close to $A$ and $(T,B)$ is uniformly hyperbolic. Let $B\in C^0_{A,\overline V}(X,SL(2,\mathbb R))$ be close enough to $A$ so that $|\epsilon(x)-1|<\delta$ for all $x\in X$ where $\epsilon(x):=\sqrt{y_{11}(x)^2+y_{21}(x)^2}$. (See (\ref{Y}).) Define $B'\in C^0(X,S')$ and $B''\in C^0(X,S')$ as previously. Then, we have $|h_{\epsilon(x)} (t(x))-r(x)|<\delta_1$ and $|g_{\epsilon(x)}(t(x))|<\delta_2$ for all $x\in X.$ (Note that $l(t(x))=r(x)$.) With such $B,$ we have $||B'-B''||_{C^0}<\epsilon'/3.$

From here, choose a $C^0-$closer $B$ (if necessary) so that $||A-B||_{C^0}<\epsilon'/3$ and $ ||B-B'||_{C^0}<\epsilon'/3.$  In conclusion, we have
$$||A-B''||_{C^0}\leq ||A-B||_{C^0}+||B-B'||_{C^0}+||B'-B''||_{C^0}<\epsilon'$$
while $B''\in C^0(X,S')$ and $(T,B'')$ is uniformly hyperbolic.

\end{proof}

\emph{Proof of Theorem 1:}

Let $f\in C^0(X,\mathbb D)$ and suppose that $f$ is not identically zero. Let $A:=W^{-1}\overline A(f,z)W\in C^0(X,S')$ and suppose that $(T,A)$ is not uniformly hyperbolic. Choose $B\in C^0(X,SL (2,\mathbb R))$, $B'\in C^0(X,S')$ and $B''\in C^0(X, S')$ as in the proof of Proposition \ref{CMV_prop1}. 

 Define $\beta:X\to \mathbb D$ as $\beta(x)=h_{\epsilon(x)}(t(x))e^{i\eta(x)}$ where 
$$\eta(x)=\tilde \theta(x)-\tau(x)+g_{\epsilon (x)}(t(x))-{\theta'}.$$
Note that we have $\beta\in C^0(X,\mathbb D)$ and 
\begin{align*}
W^{-1}\overline A (\beta(x),z)W&=W^{-1}\frac{1}{z^{1/2}\sqrt{1-|\beta(x)|^2}}\begin{bmatrix}
    z    & -\bar\beta(x)   \\
    -\beta(x)z     & 1
\end{bmatrix}W\\
&=B''(x).
\end{align*}

Since $W$ is unitary and $||A-B''||_{C^0}<\epsilon ',$ we have
$$
||WAW^{-1}-WB''W^{-1}||_{C^0}=||\overline A(f,z)-\overline A(\beta,z)||_{C^0}<\epsilon'.
$$
This implies that we may choose $\beta\in C^0(X,\mathbb D),$ which is arbitrarily $C^0-$close to $f$ and $(T,\overline A(\beta,z))$ is uniformly hyperbolic. 

Now suppose that $f\in C^0(X,\mathbb D)$ is identically zero. Then, we may choose $f'\in C^0(X,\mathbb D)$ such that it is not identically zero and arbitrarily $C^0-$close to $f.$ With $f'$, we may repeat a similar procedure as above.

Now, for $z\in \partial \mathbb  D,$ consider the set 
$$UH_z=\{f\in C^0(X,\mathbb D): (T,\overline A(f,z))\text{ is uniformly hyperbolic}\}
.$$

Then, $UH_z$ is open and dense by the previous argument. Thus, we may choose a countable dense subset $\{z_n\}$ of $\partial \mathbb D$ to conclude that for $f\in \bigcap_n UH_{z_n}$, the set $\partial \mathbb D\setminus \Sigma$ is dense. Together with the result in \cite{MR3543643}, the associated CMV operators have Cantor spectrum for a generic $f\in C^0(X,\mathbb D).$

\subsection{Proof for Jacobi matrices}

Let $f_a,f_b\in C^0(X,\mathbb R)$ with $f_a(x)>0$ for all $x\in X.$ Fix $x\in X$ and let $H_x$ be a two-sided Jacobi matrix with $a_n=f_a(T^n x)$, $b_n=f_b(T^n x).$

Define $J\subset SL(2,\mathbb R)$ as
$$J=\Big\{ \begin{bmatrix}
    t    & -\frac 1 a   \\
    a     & 0
\end{bmatrix}\in SL(2,\mathbb R) | t,a\in \mathbb R, a> 0
 \Big\}.$$
 
 Then, we have $A_{E,a,b}\in C^0(X, J).$

\begin{lem}
Let $K\subset X$ be a compact set such that $K\cap T(X)=\emptyset$ and $K\cap T^2(X)=\emptyset.$ Let $A\in C^0(X,J)$ be such that for every $x\in K,$ we have $\emph{tr} A(x)\neq 0$. Then there exist an open neighborhood $\mathcal W_{A,K}\subset C^0_{A,K}(X,SL(2,\mathbb R))$ of $A$ and continuous maps 
$$\Phi=\Phi_{A,K}:\mathcal W_{A,K}\to C^0(X,J)$$
and 
$$\Psi=\Psi_{A,K}:\mathcal W_{A,K}\to C^0(X,SL(2,\mathbb R))$$
satisfying 
$$\Psi(B)(T(x))\cdot B(x)\cdot [\Psi(B)(x)]^{-1}=\Phi(B)(x),$$
$$\Phi(A)=A\text{ and } \Psi(A)=id.$$
\label{Jacobi_lem}
\end{lem}

\begin{proof}
Let $B\in C^0_{A,K}(X, SL(2,\mathbb R))$. Let $\Phi(B)(x)=A(x)$ if $x\notin \bigcup_{i=-1}^{1}T^i(K).$ Fix $x\in K$. Let
\begin{align*}
B(T(x))B(x)B(T^{-1}(x))&=\begin{bmatrix}
    t _1   & -\frac 1 {a_1}   \\
    a_1     & 0
\end{bmatrix}\begin{bmatrix}
    p    & q   \\
    r     & s
\end{bmatrix}\begin{bmatrix}
    t_3    & -\frac 1 {a_3}   \\
    a_3     & 0
\end{bmatrix}\\
&=\begin{bmatrix}
    t_1(pt_3+qa_3)-\frac{rt_3+sa_3}{a_1}    & -\frac {pt_1} {a_3}+\frac{r}{a_1a_3}   \\
    a_1(pt_3+qa_3)     & -\frac{pa_1}{a_3}
\end{bmatrix}
\end{align*}

and let 
\begin{align*}
A(x)&=\begin{bmatrix}
    t _2   & -\frac 1 {a_2}   \\
    a_2     & 0
\end{bmatrix}.
\end{align*}

For $x\in \bigcup_{i=-1}^{1}T^i(K)$, our goal is to define $\Phi(B)(x)$ so that 
$$\Phi(B)(T(x))\Phi(B)(x)\Phi(B)(T^{-1}(x))=\begin{bmatrix}
    t_1(pt_3+qa_3)-\frac{rt_3+sa_3}{a_1}    & -\frac {pt_1} {a_3}+\frac{r}{a_1a_3}   \\
    a_1(pt_3+qa_3)     & -\frac{pa_1}{a_3}
\end{bmatrix}$$
while we have $\Phi(B)(x)\in J.$

By a simple calculation,
\begin{align*}\begin{bmatrix}
    t_1'   & -\frac{1}{a_1'} \\
    a_1'   & 0
\end{bmatrix}
\begin{bmatrix}
    t_2'   & -\frac{1}{a_2'} \\
    a_2'   & 0
\end{bmatrix}
\begin{bmatrix}
    t_3'   & -\frac{1}{a_3'} \\
    a_3'   & 0
\end{bmatrix}
=\begin{bmatrix}
    t_1'(t_2't_3'-\frac{a_3'}{a_2'})-\frac{a_2'a_3'}{a_1'}   & -\frac{t_1't_2'}{a_3'}+\frac{a_2'}{a_1'a_3'} \\
    a_1'(t_2't_3'-\frac{a_3'}{a_2'})   & -\frac{t_2'a_1'}{a_3'}
\end{bmatrix}.
\end{align*}

Set $a_1'=a_1$, $a_2'=a_2$ and $a_3'=a_3.$ We may write $t_i'$, $i=1,2,3$ as
$$t_i'=\frac{E-b_i'}{a_i'}$$
where $E,b_i'\in \mathbb R.$
Set 
$$b_2'=E-pa_2,$$
$$b_3'=E-a_3'\frac{pt_3+qa_3+\frac{a_3'}{a_2'}}{t_2'}$$
and
$$b_1'=E-\frac{a_1'a_3'}{t_2'}\Big(\frac{a_2'}{a_1'a_3'}+\frac{pt_1}{a_3}-\frac{r}{a_1a_3} \Big).$$

Note that we have $p\neq 0$ by choosing a proper neighborhood $\mathcal W_{A,K}$ of $A$ since we assume $trA(x)\neq 0$ for all $x\in K$. This, in turn, implies that $t_2'\neq 0.$ 

By setting 
$$\Phi(B)(T(x))=\begin{bmatrix}
    t_1'   & -\frac{1}{a_1'} \\
    a_1'   & 0
\end{bmatrix},
\Phi(B)(x)=\begin{bmatrix}
    t_2'   & -\frac{1}{a_2'} \\
    a_2'   & 0
\end{bmatrix}
$$
and
$$
\Phi(B)(T^{-1}(x))=\begin{bmatrix}
    t_3'   & -\frac{1}{a_3'} \\
    a_3'   & 0
\end{bmatrix},$$
we have 
$$\Phi(B)(T(x))\Phi(B)(x)\Phi(B)(T^{-1}(x))=B(T(x))B(x)B(T^{-1}(x)).$$

Let $\Psi(B)(x)=id$ for $x\notin K\cup T(K)$ and let 
$$\Psi(B)(x)=\Phi(B)(T^{-1}(x))\cdot [B(T^{-1}(x)]^{-1}$$ for $x\in K.$ Let 
$$\Psi(B)(x)=\Phi(B)(T^{-1}(x))\cdot\Phi(B)(T^{-2}(x))\cdot  [B(T^{-2}(x)]^{-1}\cdot  [B(T^{-1}(x)]^{-1}$$ for $x\in T(K).$ All properties are easy to check.

\end{proof}

By combining Lemma \ref{ABD_lemma10} and Lemma \ref{Jacobi_lem}, we obtain

\begin{prop}
Let $A\in C^0(X,J)$ be a map whose trace is not identically zero. Then there exist an open neighborhood $\mathcal W \subset C^0(X,SL(2,\mathbb R))$ of $A$ and continuous maps 
$$\Phi=\Phi_{A}:\mathcal W\to C^0(X,J)\text{ and }\Psi=\Psi_{A}:\mathcal W\to C^0(X,SL(2,\mathbb R))$$
satisfying 
$$\Psi(B)(T(x))\cdot B(x)\cdot [\Psi(B)(x)]^{-1}=\Phi(B)(x),$$
$$\Phi(A)=A\text{ and } \Psi(A)=id.$$
\label{HK_prop1}
\end{prop}

\begin{proof}
 Let $x\in X$ be such that $\emph{tr}A(x)\neq 0.$ Let $V$ be an open neighborhood of $x$ such that with $K=\bar V,$ we have $\emph{tr} A(x)\neq 0$ for $x\in K,$ $K\cap T(X)=\emptyset$ and $K\cap T^2(X)=\emptyset.$ 
 
 Let $\Phi_{A,V}:\mathcal W_{A,V}\to C^0_{A,\bar V}(X,SL(2,\mathbb R)$ and $\Psi_{A,V}:\mathcal W_{A,V}\to C^0(X,\\
 SL(2,\mathbb R))$ be given by Lemma \ref{ABD_lemma10}. Let $\Phi_{A,K}:\mathcal W_{A,K}\to C^0(X,J)\text{ and }\\
 \Psi_{A,K}:\mathcal W_{A,K}\to C^0(X,SL(2,\mathbb R))$ be given by Lemma \ref{Jacobi_lem}. Let $\mathcal W $ be the domain of $\Phi:= \Phi_{A,K}\circ \Phi_{A,V},$ and let $\Psi=(\Psi_{A,K}\circ \Phi_{A,V})\cdot \Psi_{A,V}.$ With $\Phi$ and $\Psi$, the result follows.

 \end{proof}

\begin{rem}
Let us discuss why we do not apply the above procedure in the case of CMV matrices. To make a similar argument as in the proof of Lemma \ref{Jacobi_lem}, one essential part is to construct $\Phi(B)\in C^0(X,S')$ such that 
\begin{align*}
B(T^m(x))\cdots &   B(x)\cdots B(T^{-n}(x))\\
&=\Phi(B)(T^m(x))\cdots\Phi(B)(x)\cdots \Phi(B)(T^{-n}(x))
\end{align*}
for some $m,n\in \mathbb Z_+.$
The conjugation property almost automatically follows then. As we observed in the proof of Lemma \ref{Jacobi_lem}, the construction is related to the solvability of a system of equations. 

If we write $\Phi(B)$ as in Lemma \ref{A}, the product of matrices, 
$$\Phi(B)(T^m(x))\cdots \Phi(B)(x)\cdots \Phi(B)(T^{-n}(x))$$
 may be very complicated. In addition, we have many constraints since a matrix as in Lemma \ref{A} has a very particular form while $0\leq r(x)<1$. 

Alternatively, we may consider a product of matrices as in the form of ${SU}(1,1)$. In this case, the product involves complex numbers and conjugations of those while solutions must be in the complex unit disk. Moreover, we anticipate it is unlikely that solutions are written as linear forms.

\end{rem}

Now, we are ready to prove Theorem \ref{Jacobi_thm}.
 Let $A\in C^0(X,J).$ Recall if $A$ is not uniformly hyperbolic, we may choose $B'\in C^0(X,SL(2,\mathbb R))$ such that $B'$ is arbitrarily $C^0-$close to $A$ and $(T,B')$ is uniformly hyperbolic.

\null

\emph{Proof of Theorem 2:}

For $E\in \mathbb R,$ define the set
$$UH_{E}:=\{f_b\in C^0(X,\mathbb R)| (T, A_{E,a,b}) \text{ is uniformly hyperbolic}   \}.$$

Suppose that a cocycle $(T, A_{E,a,b})$ is not uniformly hyperbolic. If $f_b(x)\neq E$ for some $x\in X,$ $(T, A_{E,a,b})$ can be approximated by a uniformly hyperbolic cocycle $(T,A_{E,a,b'})$ for some $f_{b'}\in C^0(X, \mathbb R)$ by Proposition \ref{HK_prop1}. 

Suppose that $f_b(x)=E$ for all $x\in X.$ Then, we may find $b'\in C^0(X,\mathbb R)$ such that $f_{b'}$ is arbitrarily $C^0$-close to $f_b$ and $f_{b'}(x)\neq E$ for some $x\in X.$ From $(T, A_{E,a,b'})$ we may choose a uniformly hyperbolic cocycle $(T, A_{E,a,b''})$ for some $ f_{b''}\in C^0(X,\mathbb R),$ which is arbitrarily $C^0$-close by Proposition \ref{HK_prop1}. This shows the set $UH_E$ is open and dense.  Thus, there exists a countable dense subset $\{E_n\}\subset \mathbb R$ so that for all $f_b\in \bigcap_n UH_{E_n}$, the set $\mathbb R\setminus \Sigma$ is dense.

\null

\emph{Acknowledgement.}   I am grateful to my thesis advisor David Damanik for numerous conversations and suggestions.

\end{document}